\documentclass[reqno,11pt]{amsart}
\usepackage{amssymb}
\usepackage{verbatim}
\usepackage{times}
\usepackage{setspace}
\newif\ifpdf
\ifpdf
  \usepackage[pdftex]{graphicx}
  \usepackage[pdftex]{hyperref}
\else
  \usepackage{graphicx}
\fi

\numberwithin{equation}{section}       
\theoremstyle{plain}
\newtheorem{Thm}{Theorem}[section]
\newtheorem{Prop}[Thm]{Proposition}
\newtheorem{Lem}[Thm]{Lemma}

\newtheorem{Prop-def}[Thm]{Proposition-Definition}
\newtheorem{theoalph}{Theorem}

\newtheorem{coralph}{Corollary}

\theoremstyle{definition}

\newcommand{\A}{{\mathbb{A}}}

\newcommand{\N}{{\mathbb{N}}}
\newcommand{\Q}{{\mathbb{Q}}}
\newcommand{\R}{{\mathbb{R}}}
\newcommand{\Z}{{\mathbb{Z}}}

\newcommand{\ord}{\mathrm{ord}}

\newcommand{\fm}{{\mathfrak{m}}}

\newcommand{\fB}{{\mathfrak{B}}}

\newcommand{\fX}{\mathfrak{X}}

\newcommand{\cO}{{\mathcal{O}}}

\newcommand{\cU}{{\mathcal{U}}}

\renewcommand{\div}{\mathrm{div}}

\newcommand{\spec}{\mathrm{Spec}\,}

\begin{document}
%
%

\setcounter{tocdepth}{1}

\title{Countability properties of  some Berkovich spaces}

\date{\today}

\author{Charles Favre}

\thanks{Supported by the ANR project Berko, and by the ECOS project C07E01}

\address{CNRS, \'Ecole Polytechnique, Palaiseau, 91128 Cedex France }
\email{favre@math.polytechnique.fr}


\maketitle

\tableofcontents

%
%
\section{Introduction}

Our aim is to prove some facts on the topology of a particular class
of non archi\-medean analytic spaces in the sense of Berkovich. One of the main
feature of Berkovich analytic spaces  is that their natural
topology makes them both locally arcwise connected and locally compact. However, in general
these spaces are far from being separable\footnote{a {\it separable} space is a topological space admitting a countable dense subset. We shall never use the notion of separable field extension so that no confusion should occur.}. Nevertheless, we shall prove
that analytic spaces retain some countability properties that reflect their algebraic nature. We expect these properties to be useful in a dynamical context.

\medskip

We shall work in the following setting.
We consider an  normal algebraic variety\footnote{this means $X$ is irreducible reduced of finite type and $k$ is algebraically closed. The last assumption is however inessential in our paper.} $X$ defined over a field $k$.
We fix an effective Cartier divisor $D\subset X$, and
we let $\hat{X}$ be the formal completion of $X$
along $D$. We then consider the generic fibre
$X_\eta$ of this formal scheme as defined in~\cite{thuillier}: this is an analytic space in the sense of Berkovich over $k$ endowed with the trivial norm. 
When $X = \spec A$ is affine, and $D =\{ f=0 \}$ with $f \in A$, then $X_\eta$ coincides with the set of bounded (hence $\le 1$) multiplicative semi-norms $|\cdot| : A \to \R_+$ such that $0< |f| <1$. Note that replacing $|\cdot|$ by $\nu := -\log|\cdot|$, we can define in an equivalent way $X_\eta$ as the set of  valuations  $\nu : A \to \R_+\cup \{+\infty \}$ (possibly taking the value $+ \infty$ on a non-zero element) such that $\infty>\nu(f)>0$.  In the general case, $X$ is covered by affine open subsets $U_i$, and  $X_\eta$ is obtained by patching together the $U_{i,\eta}$'s in a natural way.

Multiplying a valuation by a positive constant yields an action of $\R^*_+$ on $X_\eta$. It is therefore natural to introduce the \emph{normalized generic fiber} $ ]X [\, \subset X_\eta$ of valuations normalized
by the condition $\nu(f) =+1$. In this way,   $]X[$ can be identified with the quotient space of $X_\eta$ by 
$\R^*_+$, and  the projection map $X_\eta \to ]X[$ is a $\R^*_+$-fibration. When $D$ is complete, then the space $]X[$ is compact, see Lemma~\ref{lem:compact} below.

Our main result reads as follows.
\begin{theoalph}\label{thm:dense}
Let $X$ be a normal algebraic variety over $k$ and
$D$ be an effective  Cartier divisor.  Let $A$ be any subset of the normalized generic fiber $]X[$ of the  formal completion of $X$ along $D$. 

For any $x$ in the closure of $A$, there exists a \emph{sequence} of points $x_n \in A$ such that $x_n \to x$.
\end{theoalph}
Recall that a topological space is \emph{angelic} if any relatively $\omega$-compact set\footnote{i.e. any sequence of points in this set has a cluster point in the ambient space} is relatively compact, and for any subset $A$, any point in $\bar{A}$ is the limit of a \emph{sequence} of 
points in $A$, see~\cite{floret} for more informations. This property plays an important role
in the study of Banach spaces. A first consequence of the previous result is the following
\begin{coralph}
Let $X$ be a normal algebraic variety over $k$ and $D$ be an effective  Cartier divisor. 
Then $]X[$ is angelic. If moreover $D$ is complete then $]X[$ is sequentially compact.
\end{coralph}
In order to state yet another consequence of Theorem~A, we need to introduce some terminology.
When $X$ is affine, a divisorial valuation on the ring of regular functions $k[X]$
is a discrete valuation of rank $1$ and transcendence degree $\dim(X)-1$. Geometrically, it is given by the order of vanishing along a prime divisor in a suitable birational model of $X$.
A point  $x\in X_\eta$ is called divisorial if it is given by a divisorial valuation in the ring of regular functions of some affine chart. If we
interpret $x$ as a semi-norm, then $x$ is divisorial iff it is a norm, and 
the residue field of the completion of $k[X]$ w.r.t. the induced norm by $x$
 has transcendence degree $\dim(X)-1$ over $k$.

Finally recall that we have a natural reduction map $r_X :  X_\eta \to D$ defined in the affine case
by sending  a semi-norm $|\cdot|$  to the prime ideal $\{ |\cdot | <1\}\subset k[X]$. 
This reduction map sends a  divisorial  valuation  to its center in $X$ when it is non empty. Note that the center is automatically included in $D$  if the valuation lies in $X_\eta$. 
 
We shall also obtain
\begin{coralph}
Let $X$ be a normal algebraic variety over $k$ and $D$ be an effective  Cartier divisor. 
Then for any point $x\in X_\eta$, there exists a \emph{sequence} of divisorial points $x_n \in X_\eta$
that converges to $x$, and such that $r_X(x_n) \in\overline{r_X(x)}$ for all $n$.
\end{coralph}

Theorem~\ref{thm:dense} will be deduced from its analog on Riemann-Zariski spaces, see 
Theorem~\ref{thm:RZ-dense} below for details.
In this case, it essentially boils down to the noetherianity of the Zariski topology on schemes.

\medskip

Let us now explain how one can transfer the previous results to non-archimedean analytic spaces. The normalized generic fiber $]X[$ is not a Berkovich analytic space in a canonical way. However  in the special situation where we have a map $T : X \to \A^1_k$ such that $D = \{ T =0 \}$, then  $]X[$  turns out to be an analytic space over the non-archimedean field $k((T))$ (with the norm $\exp(-\ord_0)$).  In the case $X = \spec (k[x_i]/\mathfrak{a})$ is affine, then $\mathcal{A}= \varprojlim_n (k[x_i]/\mathfrak{a})/(T^n)$
is a $k[[T]]$-algebra topologically of finite type, and $]X[$ is the set of semi-norms on $\mathcal{A}$ whose restriction on
$k[[T]]$ is  $\exp(-\ord_0)$. In general, $]X[$ coincides with the generic fiber of the $T$-adic completion of $f$ in the sense of Berkovich (its construction in rigid geometry was previously given by Raynaud).


We shall prove:
\begin{coralph}\label{cor:angelic}
Any compact Berkovich analytic space that is defined over the field $k((T))$ is angelic. It is
in particular sequentially compact, and divisorial points are sequentially dense.
\end{coralph}
Note that this result also holds over any (non archimedean) local  field for simple reasons since any affinoid over such a field is separable.  However  when the residue field of $k$ is not countable, Berkovich analytic spaces are not metrizable.

\medskip

For curves, the density of divisorial 
valuations follows from the semi-stable reduction theorem. And the sequential compactness  is a consequence 
of the fact that any complete $\mathbb{R}$-tree (in the sense of~\cite{valtree}) is sequentially compact. We refer to~\cite{mainetti} 
for a proof.

We note that any compact analytic space over \emph{any} non-archimedean complete fields is angelic by the recent work of J.~Poineau,~\cite{poineau}. 
Even though our result is much more restrictive, 
our approach can be directly adapted to prove the sequential compactness of special compactifications of complex affine varieties in the spirit of~\cite{compact}. 

%

\medskip

The following  natural questions are related to the above results.

\medskip

\noindent {\bf Question 1}. 
Let $X$ be any Berkovich analytic space. Then  any Borel measure on $X$ is a Radon measure.

\medskip

\noindent {\bf Question 2}. 
Let $X$ be any Berkovich analytic space. Then the support of any Radon measure is separable.

\medskip

In dimension $1$,  question 2  has a positive answer. Question 1 remains open.\footnote{there 
is a gap in both proofs of this fact given in~\cite{valtree} and~\cite{baker-rumely}}.

\medskip

\noindent {\bf Acknowledgements:} we thank T. De~Pauw, A. Ducros, M. Jonsson,  J.~Kiwi, J. Nicaise, and R.~Menares for useful discussions on the material presented in this paper.  Also we deeply thank J.~Poineau for kindly informing the author about his proof of the sequential compactness of compact analytic spaces over an arbitrary field.

%
%

\section{Riemann-Zariski spaces}

Our basic references are~\cite{zariski-samuel,vaquie}.

A domain $R$ is said to be a valuation ring if it has no divisors of zero,
and for any non-zero element $x$ in the fraction field of $R$, either $x$ or $x^{-1}$
belong to $R$. Any valuation ring is local, with maximal ideal $\fm_R$
consisting of those $x \in R$ such that $x^{-1}$ does not belong to $R$.

A valuation on $R$ is a function $\nu: R\setminus\{0\} \to \Gamma$ to a totally ordered abelian group $\Gamma$
such that $\nu(ab) = \nu(a)+ \nu(b)$; and $\nu(a+b) \ge \max\{\nu(a), \nu(b)\}$.
Any valuation extends in a unique way to a valuation on the fraction field $K$ of $R$,
and the set $R_\nu = \{ a \in K, \, \nu(a) \ge 0\}$ is a valuation ring in $K$.

Conversely, to a valuation ring $R\subset K$ is associated a unique valuation
$\nu: R \to \Gamma$ up to isomorphism, where $\Gamma$ is the group obtained by moding out
$K^*$ by the multiplicative set $R\setminus \fm_R$.

In the sequel, we shall make no difference between valuation rings and valuations.

\medskip

Let $X$ be a \emph{projective normal irreducible} variety defined over a field $k$.
Abhyankhar's inequality  for a valuation $\nu$  on $k(X)$ states that 
$$
{\rm rat.rk} (\nu)+ {\rm deg.tr} (\nu) \le \dim(X)~,
$$
where  ${\rm rat.rk} (\nu)$ denotes the dimension of the $\Q$ vector space $\Gamma\otimes_\Z \Q$; 
and ${\rm deg.tr} (\nu)$ is the degree of transcendence of the residue field over $k$.
In particular, $\Gamma\otimes_\Z \Q$ is countable. We thus have
\begin{Lem}\label{lem:countable}
The value group of any valuation on $k(X)$ that is trivial on $k$ is countable.
\end{Lem}
The Riemann-Zariski space $\fX$ is the set of valuation rings in $k(X)$ containing $k$
such that the sets
$$
U(A) = \{ R, \, A \subset  R\}
$$
where $A$ ranges over all subrings of finite type of $k(X)$ that contains $k$ 
form a basis of open sets for its topology. A theorem of Zariski states: 
\begin{Thm}\label{thm:RZ-compact}
The Riemann-Zariski space is quasi-compact.
\end{Thm}
Note that $\fX$ is never Hausdorff except in dimension $0$.

\medskip

Suppose $\nu : k(X) \to \Gamma$ is a valuation with valuation ring $R_\nu$. A projective birational model of $X$
consists of  a birational map
$\phi : X' \dashrightarrow X$ from a normal projective variety to $X$. The map  $\phi$ induces an isomorphism between $k(X)$ and $k(X')$, so
 that 
$\nu$ can be viewed as a valuation on $k(X')$.
The set of closed points $x' \in X'$ such that the local ring $\cO_{X',x'}$ is included in $R_\nu$, and its maximal ideal in $\mathfrak{m}_{R_\nu}$
forms an irreducible subvariety called the center of $\nu$ in $X'$. We denote it by $C(\nu,X')$.
We shall view $C(\nu,X')$ scheme-theoretically as a (non necessarily closed) point in $X'$.
A birational model $\psi: X'' \dashrightarrow X$   dominates another one $\phi: X' \dashrightarrow X$ iff $\mu := \phi^{-1} \circ 
\psi$ is regular. If $X''$ dominates $X'$,  then $\mu (C(\nu,X'')) = C(\nu, X')$.

\medskip

Consider the category $\fB$ of all projective birational models of $X$ up to natural isomorphism, each model 
endowed with the Zariski topology. It is an inductive set for the relation of domination introduced before.
We may thus consider the projective limit  of all birational models of $X$, that is
$\varprojlim_{X'\in \fB} X'$ endowed with the projective limit topology. 
Concretely, a point in $\varprojlim_{X' \in \fB} X'$ is a collection of irreducible subvarieties $Z_{X'}\subset X'$
for each projective birational model $X'$ such that $\mu (Z_{X''}) = Z_{X'}$ if $X''$ dominates $X'$.

For a given valuation $\nu$, we may attach the collection $\{C(\nu,X')\}_{X'\in \fB}$. This defines a map
from $\fX$ to $\varprojlim_{X' \in \fB} X'$. Conversely, given a point $Z= \{Z_{X'}\}_{X'\in \fB}$ in the projective limit, we define 
the subset $R_Z$ of $k(X)$ of those meromorphic functions that are regular at the generic point of
$Z_{X'}$ for any birational model $X'$ of $X$. It is not difficult to check that $R_Z$ is a valuation ring.

These two maps are inverse one to the other. More precisely, one has the following fundamental result again due to Zariski:
\begin{Thm}\label{thm:RZ-caract}
The natural map $\fX \to \varprojlim_{X' \in \fB} X'$ given by $\nu \mapsto\{ C(\nu,X')\}_{X'\in \fB}$
induces a homeomorphism.
\end{Thm}
This result has the following useful consequence. For any model $X'$, and any 
 Weil divisor $D'$ in $X'$, let $\cU(X',D')$ be the set of all 
 valuation rings whose center is not included in $D'$.
 Then the  collection of all sets of the form $\cU(X',D')$
gives a basis for the topology of $\fX$.

Note that if we cover 
$X'\setminus D'$ by affine charts $Y_i$, then $X'\setminus Y_i := D_i$ is a divisor since $X'$ is normal, and
$\cU(X',D') = \cup\, \cU(X',D_i)$. Thus the collection of all $\cU(X',D')$'s such that $X'\setminus D'$ is affine also forms a basis for the topology 
of $\fX$.
\begin{Lem}\label{lem:cvg}
A sequence of valuations $\nu_n$ converges to $\nu$ in $\fX$ iff
for any $f\in R_\nu$, we have $f\in R_{\nu_n}$ for $n$ large enough.
\end{Lem}
\begin{proof}
Suppose first $\nu_n \to \nu$, and pick $f \in R_\nu$. Choose a model $X'$ such that 
$f$ is regular on $X'$, and let $D'$ be the set of poles of $f$.  Since $f$ belongs to $R_\nu$, the center of  $\nu$ cannot be included in $D'$, and $\nu\in\cU(X',D')$. By assumption $\nu_n\to \nu$ hence $\nu_n\in\cU(X',D')$ for $n$ large enough. This implies the center of $\nu_n$ not to be included in $D'$, and 
we conclude that $f\in R_{\nu_n}$.

Conversely, let us assume that for any $f\in R_\nu$, we have $f\in R_{\nu_n}$ for $n$ large enough.
The collection of open sets $\cU(X',D')$ with 
$X'$ a birational  model of  $X$, and $D'\subset X'$ a divisor forms a basis for the topology on $\fX$.
Therefore proving the convergence of $\nu_n$ to $\nu$ is equivalent to show $\nu_n \in \cU(X',D')$
for $n$ large enough if $\nu \in \cU(X',D')$.
As noted above,  we may assume $Y' = X'\setminus D'$ is affine.
Choose a finite set of regular functions $f_i$ on the affine space $Y'$ that generate $k[Y']$.
Since the center of $\nu$ in $X'$ is not included in $D'$, we have $\nu(f_i)\ge0$ for all $i$.
By assumption, we get $\nu_n(f_i)\ge0$ for all $i$ and all large enough $n$.
But then the center of $\nu_n$ cannot be in $D'$, hence $\nu_n \in  \cU(X',D')$. This concludes the proof.
\end{proof}
Finally we shall use several times the
\begin{Lem}\label{lem:blw}
Suppose $\mathfrak{I}$ is a coherent sheaf of ideals on an irreducible  normal variety $X$.
Then there exists a regular birational map $\phi: X' \to X$ such that 
$X'$ is normal and $\mathfrak{I}\cdot \cO_{X'}$ is locally principal (ie. invertible).
\end{Lem}
Take $X'$ to be the normalization of the blow-up of $\mathfrak{I}$, and use the universal property
of blow ups.

%
%

\section{Countability properties in Riemann-Zariski spaces}

In this section, $X$ is a \emph{ normal projective algebraic variety} defined over a (non necessarily algebraically closed) field $k$.
Our aim is to prove
\begin{Thm}\label{thm:RZ-dense}
Let $A$ be any subset of $\fX$. Then for any $\nu$ in the closure of $A$, 
either $\nu \in \overline{\{\mu\}}$ for some $\mu \in A$, or
one can find a sequence of valuations $\nu_n \in A$ such that $\nu_n \to \nu$.
\end{Thm}
The proof relies on the following lemma.
Recall that given an affine variety $Y$, a valuation  $\nu : k[Y] \to \Gamma$ and $\gamma \in \Gamma$, the  valuation ideal is defined by 
$$I(\nu,\gamma):= \{f,\, \nu(f) \ge \gamma \}\subset k[Y]~.$$ 
\begin{Lem}\label{lem:val-ideal}
Let $Y$ be an irreducible affine variety.  Suppose $\nu :k[Y] \to \Gamma$ is a valuation
whose center in $Y$ is non-empty.
Pick $\gamma \in \Gamma$, with $\gamma>0$, and consider a proper modification
$\mu: Y' \to Y$ with $Y'$ normal such that
$I(\nu,\gamma)\cdot \cO_{Y'}= \cO_{Y'}(-D)$
for some effective Cartier divisor  $D$.

Then $C(\nu,Y')$ is included in the support of $D$,
and for any $f \in k[Y]$ such that $\nu(f) = \gamma$,
the divisor $\div(f\circ \mu)\subset Y'$ is equal to $D$ at
the generic point of $C(\nu,Y')$.
\end{Lem}

\begin{proof}
Suppose there exist a closed point $p\in Y'$, and
a regular function $f \in \cO_Y$ such that  $\nu(f) = \gamma$, and $f\circ \mu$
does not vanish  at $p$. Then the lift of $f^{-1}$ is regular at
$p$ but does not belong to the valuation ring at $\nu$. Whence $p \notin C(\nu,Y')$.
This proves the first claim.

Next pick an element $g$ in the valuation ideal so that $\nu(g) = \gamma$, and the strict transform
of $g$ in $Y'$ does not contain $C(\nu,Y')$.
We may assume that the exceptional part of the divisor $\{ g \circ \mu =0 \}$ is equal to $D$.
At a generic point $p \in C(\nu,Y')$, write $f \circ \mu = \tilde{f} \times \hat{f}$
where $\tilde{f} =0$ defines the strict transform of $\{ f =0 \}$ in $Y'$, and $\{\hat{f} =0 \}$
is supported on the exceptional divisor. Since $\gamma = \nu(f)$, we have $\div(\hat{f}) \ge \div(g)$
so that the quotient $\hat{f}/g$ is regular at $p$.
Look at the equation:
$$\gamma = \nu(f) = \nu(g) + \nu(\hat{f}/g) + \nu(\tilde{f})= \gamma + \nu(\hat{f}/g) + \nu(\tilde{f})~.$$
Since $\tilde{f}$ and $\hat{f}/g$ are regular at $p$ that belong to the center of $\nu$, we get
$\nu(\tilde{f})\ge0$, and $\nu(\hat{f}/g) \ge 0$, whence $\nu(\tilde{f})= \nu(\hat{f}/g) = 0$, and 
$\tilde{f}, \hat{f}/g$ are both non-zero at $p$.
This proves the claim.
\end{proof}

\begin{proof}[Proof of Theorem~\ref{thm:RZ-dense}]
There is no loss of generality in assuming $\nu \notin \overline{\{\mu\}}$ for any $\mu \in A$.
Let $Y\subset X$ be an affine chart intersecting the center of $\nu$, and
write $\nu: k[Y] \to \Gamma$ for some totally ordered group $\Gamma$. By Lemma~\ref{lem:countable}, $\nu(\cO_Y)$ is countable.
By Lemma~\ref{lem:blw}, we may produce a sequence
 $ Y_{n+1} \mathop{\longrightarrow}\limits^{\pi_n}Y_n  \mathop{\longrightarrow}\limits^{\pi_{n-1}} ...  \mathop{\longrightarrow}\limits^{\pi_1} Y$ of birational models such that  for any $\gamma \in \Gamma$,
 the valuation ideal  sheaf $I(\nu, \gamma)\cdot \cO_{Y_n}$ is  locally principal
 for all $n$ large enough.

For any subset $B \subset A$, we introduce the Zariski closed set $$C_n(B) := \overline{\cup_{\mu \in B} C(\mu,Y_n)}~.$$
%
%
%
\begin{Lem}\label{lem:interm-constr}
There exists a countable subset $A'\subset A$ and irreducible subvarieties $Z_n \subset Y_n$ 
such that
\begin{enumerate}
\item
$\pi_{n-1}(Z_n) = Z_{n-1}$ for all $n$;
\item
the restriction maps $\pi_{n-1} : Z_n \to Z_{n-1}$ are birational;
\item
 for all $n$, $C(\nu,Y_n) \subset Z_n$;
\item
for all $n$,
$C_n(A') = Z_n$.
\end{enumerate} 
\end{Lem}
Pick any enumeration $\{\nu_m\}$ of the elements of $A'$, and
consider a countable field $K$ such that all varieties $Z_n, C(\nu_m,Y_n), C (\nu, Y_n)$ are defined over $K$. 

Let us introduce the following terminology. A pro-divisor $W  = \{ W_n\}$ is a collection of (possibly zero) reduced divisors $W_n$ in $Z_n$  defined over $K$, such that $\pi_n(W_{n+1}) = W_n$ for all $n$, there exists an $N$ for which $W_N$ is irreducible, and $W_n = (\pi_{n-1} \circ ... \circ \pi_N)^{-1} (W_N)$ for all $n\ge N$. We call the minimal $N$ having this property the height of $W$ and denote it by $h(W)$.
Since for each $N$  the set of prime divisors of $Z_N$ and defined over $K$ is countable, we may  find a sequence $W^j$ enumerating all  pro-divisors.

For any valuation $\mu\in A'$ and any pro-divisor $W$, we write $\mu \subset W$  if
$C(\mu, Y_n) \subset W_n$ for all $n$. This condition is equivalent to impose $C(\mu, Y_N) \subset W_N$ for $N= h(W)$. 
\begin{Lem}\label{lem:exploit}
Suppose $\mu \in A'$, and 
$\nu \notin \overline{\{\mu\}}$. Then one can find an integer $j$ such that 
$\mu\subset W^j$.
\end{Lem}
We now define by induction a sequence of valuations $\tilde{\nu}_m$ and of integers $j_m < j_{m+1}$ such that
\begin{itemize}
\item $\tilde{\nu}_l \subset W^{j_l}$;
\item $\tilde{\nu}_l \not\subset W^{j}$ for all $j < j_l$.
\end{itemize}
To do so we proceed as follows. Set $A_j = \{ \mu \in A', \, \mu \subset W^j\}$.
Set $j_1 = \min \{ j, \, A_j \neq \emptyset \}$, and define recursively $j_{m+1} = \min \{ j > j_m , \, A_j \setminus \cup_{l<j} A_l \neq \emptyset \}$. 

Let us justify the existence of $j_m$ for all $m$. By contradiction, assume that 
$A_j \setminus \cup_{l<j} A_l  = \emptyset$ for all $j > j_m$. Since by Lemma~\ref{lem:exploit} any valuation in $A'$ belongs to some $W^j$,  we have $A' = \cup_{j \ge 0} A_j$.
As a consequence,  we conclude that $A' \subset \cup_{j \le j_m} A_j$
which  forces all
valuations in $A'$ to have a center included in some fixed divisor of $Z_1$. This contradicts 
property (4) of Lemma~\ref{lem:interm-constr}.

Finally we pick any sequence $\tilde{\nu}_l \in A_{j_m}$.

\medskip

We now prove that $\tilde{\nu}_n$ converges to $\nu$. Pick $f \in R_\nu$, write $f = \frac{g}{h}$ with $g,h \in k[Y]$. Pick $N$ sufficiently large such that $I(\nu(g),\nu)$ and $I(\nu(h),\nu)$
are locally principal in $Y_N$. By Lemma~\ref{lem:val-ideal},
the lift of $f$ to $Y_N$ is regular at the generic point
of $C(\nu,Y_N)$.

Let $Z$ be the set of poles of $f$ in $Y_N$. 
Since $f$ is regular along $C(\nu, Y_N)\subset Z_N$,  and $Z_N$ is irreducible, $Z \cap Z_N$ is a  divisor of $Z_N$ that is possibly empty. If so, then  $f$ is regular at the generic point of $C(\tilde{\nu}_l,Y_N)$ for all $l$ which implies $f \in R_{\tilde{\nu}_l}$ for all $l$. Otherwise,  $Z \cap Z_N$ determines a unique pro-divisor say $W^J$ for some $J$ such that 
$W^J_N = Z \cap Z_N$, and $W^J_n = (\pi_{n-1} \circ ... \circ \pi_N)^{-1}(W^J_N)$ for all $n\ge N$.

By construction,  for any integer $l$ such that $j_l >J$, we have $\tilde{\nu}_l \not\subset W^J$.
In other words, we have $C(\tilde{\nu}_l,Y_N) \not\subset W^J_N = Z \cap Z_N$.
We have thus proved that $f$ is regular at the generic point of $C(\tilde{\nu}_l,Y_N)$ for $l$ large enough. This implies $f \in R_{\tilde{\nu}_l}$ for $l$ large enough, and concludes the proof.
\end{proof}

\begin{proof}[Proof of Lemma~\ref{lem:exploit}]
Pick any valuation $\mu\in A'$.  Assume that the center of $\mu$ in $Y_n$ is equal to $Z_n$ for all $n$. We need to prove that $\nu \in \overline{\mu}$, i.e. $R_\nu \subset R_\mu$.

Pick any $f \in R_\nu$. As in the proof above, we can find an $N$ such that $f$ is regular at the generic point of the center of $\nu$ in $Y_N$. Hence $f$ is regular at the generic point of $Z_n$.
Since the latter is the center of $\mu$, we conclude that $f \in R_\mu$.
\end{proof}

\begin{proof}[Proof of Lemma~\ref{lem:interm-constr}]
We first construct varieties $Z_n$ in $Y_n$ satisfying the first three properties.
For that purpose, let us introduce the set
$\mathcal{Z}$ be the set of all sequences $Z_\bullet : = \{Z_n\}_n$ of subvarieties of $Y_n$ such that $Z_n$ is an  irreducible component of $C_n(A)$ and
$\pi_n(Z_{n+1}) \subset Z_n$. Since we have $\pi_n (C(\mu, Y_{n+1})) = C(\mu, Y_n)$ for any valuation, note that $\pi_n (Z_{n+1}) = Z_n$ for all $n$.

For any fixed $n$, the set of irreducible components of $C_n(A)$ is finite of cardinality $d(n)$.
The set of sequences of irreducibles components of $C_n(A)$ is thus in natural bijection with 
$\Sigma = \prod_{n=1}^\infty \{ 1, ..., d(n)\}$.  For the product topology, $\Sigma$ is a compact (totally disconneted) space, and $\mathcal{Z}$ is a closed (hence compact) subset of $\Sigma$.

Now consider $\mathcal{K}_n := \{ Z_\bullet \in \mathcal{Z}, \, Z_n\supset C (\nu,Y_n)\}$. 
If $Z_\bullet$  belongs to $\mathcal{K}_n$, then we have $Z_{n-1} = \pi_{n-1}(Z_n) \supset \pi_{n-1} (C(\nu, Y_n)) = C(\nu ,Y_{n-1})$.  Hence $\mathcal{K}_n$ forms a decreasing sequence
of non empty compact subsets of $\mathcal{Z}$. The intersection $\cap_n \mathcal{K}_n$ is thus non empty, and we may pick  $Z_\bullet \in \cap_n \mathcal{K}_n$.

This sequence of varieties $Z_n$ satisfies the properties (1) and (2) of the lemma. 
Since $\pi_{n-1}(Z_n) = Z_{n-1}$, the dimension of $Z_n$ is increasing hence stationnary.
Replacing the sequence $Y_n$ by $Y_{n+N}$ with $N$ large enough, we may thus assume
$\pi_{n-1}:Z_n \to Z_{n-1}$ is birational for all $n$, so that (3) also holds.

For each $n$, we now  define $A_n := \{ \mu \in A, \, C(\mu, Y_n) \subset Z_n\}$.
We claim that for all $k\le n$, we have $C_k(A_n) = Z_n$.

Since $Z_n$ is an irreducible component of $C_n(A)$, it is clear that $C_n(A_n) = Z_n$.
Now $\pi_{n-1}$ is birational hence closed, which implies
$\pi_{n-1}  (\bar{S}) = \overline{\pi_{n-1}(S)}$ for any subset $S \subset Y_n$.
We infer 
\begin{multline*}
Z_{n-1} = \pi_{n-1} (Z_n) =
\pi_{n-1} ( C_n(A_n)) = \overline{\pi_{n-1}(\cup_{\mu\in A_n} C(\mu, Y_n))}
=\\
\overline{ \cup_{\mu \in A_n} C(\mu, Y_{n-1})}
= C_{n-1}(A_n)~,
\end{multline*}
and we conclude by a descending induction.

Finally we construct the subset $A'$.

First we shall construct countable subsets of $A$ with special properties.
 To any countable subset $\mathcal{N}$ of $A_n$, we attach the integer
$d_n(\mathcal{N}) = \dim C_n(\mathcal{N})$, and let $r_n(\mathcal{N})$ be the number of irreducible components of $C_n(\mathcal{N})$. Suppose one can find a sequence of countable sets $\mathcal{N}^j$ such that  $d_n(\mathcal{N}^j)$ is constant, and 
$r_n(\mathcal{N}^j) \to \infty$. Then 
$d_n(\mathcal{N} ) > d_n(\mathcal{N}^j)$ for $\mathcal{N} := \cup_j \mathcal{N}^j$. This shows that there exists a countable subset $\mathcal{N}$ of $A_n$ maximizing the pair $(d_n(\mathcal{N}), r_n(\mathcal{N}))$ for the lexicographic order on $\N^2$. We claim that $d_n(\mathcal{N}) = \dim Z_n$.

Indeed if it were not the case,  and since  $\cup_{A_n} C(\mu, Y_n)$ is Zariski dense in $Z_n$, 
then we could find a valuation $\mu \in A_n$ such that
$C(\mu, Y_n) \not\subset C_n(\mathcal{N})$ which would contradict the maximality of 
$(d_n(\mathcal{N}), r_n(\mathcal{N}))$.
Since $Z_n$ is irreducible, for each $n$ we have found 
a countable subset $\mathcal{N}^n \subset A_n$ such that  $C_n(\mathcal{N}^n) = Z_n$.
We conclude the proof by setting $A' := \cup_n \mathcal{N}^n$.
\end{proof}

\section{Proof of the main results}

In this last section, we explain how Theorem~\ref{thm:dense} can be deduced from  Theorems~\ref{thm:RZ-dense}. Proofs of  Corollaries~A and~B are given at the end of this section.


\subsection{The projection of the Riemann-Zariski space to the generic fiber}

\begin{Prop}\label{lem:projec}
Pick any projective normal variety $Y$, and any effective Cartier divisor $E$ in $Y$.
Denote by $\mathfrak{Y}(E)$ the subset of the Riemann-Zariski space  $\mathfrak{Y}$ of $Y$ consisting of those valuations whose center in $Y$ is included in $E$. 

Then there exists  a surjective and continuous map $\Pi: \mathfrak{Y}(E) \to ]Y[$
such that for any $\nu$, the center of $\nu$ in $Y$ is included in  $r_Y(\Pi(\nu))$ (with equality when $\nu$ is divisorial), and any divisorial point in $]Y[$ has a preimage in $ \mathfrak{Y}(E)$ which is divisorial too. 
\end{Prop}
This result is well-known but we give a proof for sake of completeness.

\begin{proof}
To construct  $\Pi$,  pick a finite collection of  affine charts $Y_j\subset Y$, each intersecting $E$, and such that
$\cup_j Y_j \supset E$. For each $j$, pick an equation $f_j\in k[Y_j]$ of $E$. Recall that 
$]Y_j[$ is  the set of all 
multiplicative semi-norms $|\cdot|: k[Y_j]\to \mathbb{R}_+$ such that $- \log|f_j| = +1$.
Then $]Y[$ is  the disjoint union of the $]Y_j[$'s patched together is a natural way.
For any valuation $\nu\in\mathfrak{Y}(E)$ with valuation ring $R_\nu$, 
we take $j$ such that the center of $\nu$ in $Y_j$ is non-empty, and 
define  the function $\Pi(\nu)\equiv|\cdot|_\nu: k[Y_j]\to \mathbb{R}_+$ by setting
$$\Pi(\nu)(f) = - \log \left(
\sup\{ p/q\in\mathbb{Q}_+, \text{ such that }  \, f^q/f_j^p \in R_\nu, \, p \in \N, \, q \in \N^*\} \right)$$ for any $f\in k[Y_j]$. We claim this is a multiplicative semi-norm. To simplify notation, we shall
 work with  $\mu(f) = \exp ( - \Pi(\nu)(f))$. We shall use repeteadly the fact that $g \in k(Y)$ belongs to the valuation ring $R_\nu$ iff $g^n \in R_\nu$ for some $n\in \N^*$.

Fix $f \in k[Y_j]$, and let $I(f) = \{ p/q \in \mathbb{Q}_+, \text{ such that }
\, f^q/f_j^p \in R_\nu\}$. Then $I(f)$ is a segment containing $0$.
Indeed pick $p/q > p'/q' \in I(f)$.
Then $ (f^q/f_j^p)^{q'} = (f^{q'}/f_j^{p'})^q \times f_j^{qp'-pq'} \in R_\nu$, 
and  $ (f^q/f_j^p)\in R_\nu$.

Next assume $f^q/f_j^p, g^q/f_j^p \in R_\nu$. The same argument as before shows that  for all $0\le i\le q$
we have $(f^{q-i}g^i)/f_j^p \in R_\nu$.
Whence $(f+g)^q/f_j^p \in R_\nu$. This shows $\mu(f+g) \ge \min \{ \mu(f), \mu(g) \}$.

Finally if $f^q/f_j^p, g^{q'}/f_j^{p'} \in R_\nu$, then $(fg)^{qq'}/f_j^{pq'+p'q}\in R_\nu$ so that 
 $\mu(fg) \ge \mu(f) + \mu(g)$. Conversely, if $p_0/q_0 > \mu(f) + \mu(g)$ then we may find two rational numbers 
 $p/q>\mu(f)$,  $p'/q'>\mu(g)$ such that $p_0/q_0 = p/q + p'/q'$. Since  $f_j^p/f^q, f_j^{p'}/g^{q'} \in R_\nu$, we get $f_j^{p_0}/(fg)^{q_0}\in R_\nu$. This proves  $\mu(fg) =\mu(f) + \mu(g)$, and concludes the proof that $\Pi(\nu)$ is
 a multiplicative semi-norm.

\smallskip

To see the continuity of $\Pi$ we pick a (net) $\nu_n$ converging to a valuation $\nu$ in the Riemann-Zariski space, and we pick 
$f\in k[Y_j]$. Take $p/q \in \Q_+$ such that $f^q/f_j^p \in R_\nu$, i.e. $\Pi(\nu)(f) \le -\log ( p/q)$. Then for an index $n$ large enough $f^q/f_j^p \in R_{\nu_n}$ too, hence
$\varlimsup_n \Pi(\nu_n)(f)\le \Pi(\nu)(f)$. Conversely, suppose $f_j^p/f^q \in R_\nu$ (i.e. $\Pi(\nu)(f) 
\ge -\log (p/q)$). The same argument shows
$\varliminf_n \Pi(\nu_n)(f)\le \Pi(\nu)(f)$. 

\smallskip

Let $Z\subset E$ be the center of a valuation $\nu \in \mathfrak{Y}(E)$, and suppose $Z\cap Y_j \neq\O$.
Then $r_Y(\Pi(\nu))$ is described by the prime ideal of functions $f \in k[Y_j]$ such that 
$\Pi(\nu)(f) <1$, or in an equivalent way such that $f^q/f_j^p \in R_\nu$ for some $p/q>0$.
But $f_j\in\fm_{R_{\nu_j}}$, hence $f$ too, and $r_Y(\Pi(\nu))\supset Z$.

\smallskip

By construction, any rank $1$ valuation $\nu\in\mathfrak{Y}(E)$ is mapped to the unique norm $|\cdot|$ on $k(Y)$ such that $R_\nu = \{ f, \, \log|f| \le 1 \}$ and $ - \log |f_j | = +1$ in some chart. Since norms are dense in $]Y[$, and $\Pi$ is continuous, we infer that $\Pi$ is surjective.
We complete the proof by noting that $\Pi$ maps divisorial valuations to divisorial points.
\end{proof}

We also collect the following result for later reference
\begin{Lem}\label{lem:compact}
Pick any projective normal variety $Y$, and any effective Cartier divisor $E$ in $Y$ whose support is {\it complete}. Then $\mathfrak{Y}(E)$ is quasi-compact, and $]E[$ is compact.
\end{Lem}
\begin{proof}
It is a theorem of Zariski~\cite{zariski-samuel} that $\mathfrak{Y}$ is quasi-compact. The result follows since $\mathfrak{Y}(E)$ is a closed subset of $\mathfrak{Y}$.
\end{proof}


\subsection{Proof of Theorem~\ref{thm:dense}}

Recall our assumption: $X$ is a normal algebraic variety, $D$
is an effective Cartier divisor in $X$, and $A$ is any subset of $]X[$.

We first cover $X$ by finitely many affine open sets $X_i$, and write $D_i = D \cap X_i$.
We let $]X_i[$ be the normalized generic fiber of the formal completion of $X_i$ along $D_i$.
This is an open (dense) subset of $]X[$, and $\cup_i \, ]X_i[\,  = \,]X[\,$.

For each $i$, we fix a projective (normal) compactification $X_i \subset \bar{X}_i$. We write $\bar{D}_i$ for the closure of $D$ in $\bar{X}_i$. It is a priori only a Weil divisor, but we may suppose it is Cartier by  taking the normalized blow up associated to the ideal defining $\bar{D}_i$.

Now pick any point $x$ in the closure of $A$ in $]X[$. It belongs to $]X_i[$ for some $i$.
Since $]X_i[$ is open in $]X[$, the point $x$ also lie in the closure of $A_i =  A\cap \,]X_i[\,$.
Apply Proposition~\ref{lem:projec} to $Y = \bar{X}_i$, and $E = \bar{D}_i$. This yields a continuous and surjective map $\Pi: \mathfrak{Y}(E) \to\, ]\bar{X}_i[\,\supset \,]X_i[\,$. Pick any valuation $\nu\in \Pi^{-1}(x) \subset \mathfrak{Y}(E)$. 
Since $\Pi$ is continuous, $\Pi^{-1}(x)$ is included in the closure of $\Pi^{-1}(A_i)$.
Theorem~\ref{thm:RZ-dense} implies the existence of a sequence of valuations $\nu_n \in \Pi^{-1}(A_i)$ converging to $\nu$ so that  $\Pi(\nu_n)\to x$.


\subsection{Proofs of Corollaries~A and~B}

As in the previous section, $X$ is a normal algebraic variety, and $D$ is an effective Cartier divisor.

\medskip

We first prove Corollary~A. 
Pick any subset $A$ of $]X[$. First take $x$ in the closure of $A$.
By Theorem~A, there exists a sequence $x_n \in A$ such that $x_n \to x$.
Now suppose $A$ is relatively $\omega$-compact. We need to show that it is relatively 
compact in $]X[$. Consider $\bar{X}$ any complete algebraic variety that contains $X$ as a Zariski dense subset and such that the closure $\bar{D}$ of $D$ in $\bar{X}$ is still Cartier. Such a space is given by Nagata's theorem, see~\cite{conrad} for a modern account.

Let $A'$ be the closure of $A$ in $]\bar{X}[$. By Lemma~\ref{lem:compact}, $]\bar{X}[$ hence $A'$ are compact.  Now suppose by contradiction that we can find 
$x \in A' \setminus \,]X[$.  Theorem~A applied to $A$ in $\bar{X}$ implies the existence of a sequence $x_n \in A$ such that $x_n \to x$. Since $A$ is relatively $\omega$-compact, $x_n$ admits a cluster point in $]X[$ which is absurd. Thus $A'$ is included in $]X[$ which proves that $A$ is relatively compact.

These concludes the proof of Corollary~A.

\medskip

We now prove Corollary~B. Pick any $x \in X_\eta$.
Note that for any closed subset $C$ of $D$ the preimage $r_X^{-1}(C)$ is
open. In particular, $r_X^{-1} ( \overline{ r_X(x)})$ is an open neighborhood of $x$ in $X_\eta$.
We claim that divisorial norms are dense in $X_\eta$ (hence in $]X[ \, \cap r_X^{-1} ( \overline{ r_X(x)})$). We conclude by applying Theorem~A.

To justify our claim, we proceed as follows.
Just as in the proof of Theorem~A, we cover $X$ by affine charts $X_i$
and take projective compactifications $\bar{X}_i$ of $X_i$. 
The set of divisorial valuations on $k(\bar{X}_i)$ that are centered in $\bar{D}_i$
is dense in the subset of the  Riemann-Zariski space of $\bar{X}_i$
of valuations centered in $\bar{D}_i$. By Proposition~\ref{lem:projec}, this shows
divisorial norms are dense in $]\bar{X}_i[$. Since $]X_i[$ is open in $]\bar{X}_i[$,
divisorial norms are also dense in $]X_i[$, hence in $]X[$ as required.


\subsection{Proofs of Corollary~C}

Since a compact analytic space over $k((T))$ is covered by finitely many affinoids, and any affinoid is a closed
subset in a ball of a suitable dimension,  it is sufficient to treat the case
of the unit  ball in $\A^n_{k((T))}$.

Let $D$ is the hyperplane $\{x_1=0\}$  in the affine space 
$(x_1, ..., x_{n+1}) \in X:= \A_{k}^{n+1}$. The normalized
generic fiber $]X[$ of the formal completion of $X$ along $D$
is the set of multiplicative semi-norms $|\cdot|: k[x_1, ..., x_{n+1}] \to \R_+$ 
trivial on $k$, and  such that  $|x_1| = e^{-1}$ and $|x_i| \le 1$ for all $i\ge 2$.  
This is precisely the unit ball in $X = \A_{k((T))}^{n}$. 
By Theorem~A, the unit ball in any dimension is thus angelic.


%
%

\end{document}